\documentclass[]{article}
\usepackage{rotating}
\usepackage[utf8]{inputenc}
\usepackage[english]{babel}
\usepackage{graphicx}
\usepackage{amsthm}
\newtheorem{theorem}{Theorem}
\newtheorem{corollary}{Corollary}
\newtheorem{lemma}{Lemma}

\title{Largest polyomino with no four cells equally spaced on a straight line}
\author{
\textsc{Jan Kristian Haugland}\\
\texttt{admin@neutreeko.net}}

\begin{document}

\maketitle

\begin{abstract}
\noindent The maximal number of cells in a polyomino with no four cells equally spaced on a straight line is determined to be 142. This is based on several partial results, each of which can be verified with computer assistance. 
\end{abstract}

\section{Definitions}
A \textit{polyomino} is an object in the plane formed by joining one or more unit squares (called \textit{cells}) edge to edge. It can be viewed as a graph with the cells as vertices and an edge joining two vertices if the corresponding cells are adjacent.

A polyomino is called \textit{admissible} if no four cells (i.e., their centres) are equally spaced on a straight line. A \textit{path} is a polyomino that either consists of a single cell, or contains two cells of degree 1 (called \textit{endpoints}) while all the remaining cells have degree 2.

Suppose a subset of the cells of a polyomino $P$ form a path $Q$ with endpoints $A$ and $B$. If the graph distance between $A$ and $B$ in $P$ is equal to the graph diameter $d$ of $P$, and $Q$ consists of $d+1$ cells, then $Q$ is said to be a \textit{maximal} path in $P$. The \textit{radius} of $P$ with respect to $Q$ is then the minimal value of $r$ such that for any cell $C$ in $P$, there exists a cell $D$ in $Q$ for which the graph distance between $C$ and $D$ is at most $r$.

A \textit{loop} is a polyomino for which all the cells have degree 2.

\section{Introduction}
The objective of this paper is to outline a verification of the following result:

\begin{theorem}
	An admissible polyomino may have at most 142 cells.
\end{theorem}

It is straightforward to generate all admissible paths with a computer program, and this will serve as the basis of the verification. Suppose, temporarily, that we only considered a special type of paths in which we could traverse the cells from one endpoint to the other by always moving either East or North, say, from one cell to the next. Instead of the polyomino itself, we could view the path as a sequence on the symbols E, N (East, North). The requirement that no four cells are equally spaced on a straight line would then be equivalent to requiring that no three consecutive "blocks" of symbols are permutations of each other. Dekking \cite{Dekking:1979} has shown that such a sequence could be infinitely long if no \textit{four} consecutive blocks are permutations of each other, and mentions that the case with three blocks is easily checked to only have finite solutions. But in our more general case, the paths can have as many as 120 cells.

\section{Analysis}
The basic idea in verifying Theorem 1 is to go through the admissible paths (or a subset of them, as we shall see later), and for each one, either find the largest admissible polyominoes that contain it as a maximal path, or find an upper bound for their size.

As a polyomino is built one cell at a time from a maximal path, it can be useful to keep track of the graph diameter of the intermediate polyominoes. Therefore, we start with a result on the existence of loops in large admissible polyominoes.

\begin{lemma}
	An admissible polyomino with at least 67 cells can not contain any loop of length greater than 4.
\end{lemma}

\begin{proof}[Verification]
	Table 1 shows all admissible loops (up to isometry) of length greater than or equal to 8, together with the maximal number of cells that may be added (so that the resulting polyomino remains admissible), and the maximal total number of cells. Each loop is given by a set of directions for moving from one cell to the next around the loop, with E, N, W and S representing East, North, West and South respectively.
\end{proof}

\begin{sidewaystable}
	\begin{tabular}{|l|c|c|c|}
		\hline
		 & & Max. no. of & Max. total \\
		Loop & Length & extra cells & no. of cells \\
		\hline
		EENNWWSS & 8 & 58 & 66 \\
		EENENNWWSWSS & 12 & 14 & 26 \\
		EENEENNWNWWSSWSS & 16 & 7 & 23 \\
		EENENNWNWWSWSSES & 16 & 4 & 20 \\
		EENEENNWNWWSWWSSES & 18 & 6 & 24 \\
		EENEENNWWNWWSWSSES &18 & 4 & 22 \\
		EENEENNWNNWWSWWSSESS & 20 & 1 & 21 \\
		EENEENNWNNEENEENNWNNWWSWSWWNWWSWSSESESSWSSES & 44 & 8 & 52 \\
		EENEENNWNNEENEENNWNWWSWWNWNWWSSWSSESESSWSSES & 44 & 10 & 54 \\
		EENEENNWNNEENEENNWNNWWSWSWWNNWWSSWSSESESSWSSES & 46 & 6 & 52 \\
		EENEENNWNNEENEENNWNNWWSSWWNWNWWSSWSSESESSWSSES & 46 & 6 & 52 \\
		EENEESESEENENNWNNENENNWNWWSWWNWNWWSWSSESSWSWSSES & 48 & 16 & 64 \\
		EENEESESEENENNWNWNNENNWNWWSWWNWNWWSWSSESESSWSSES & 48 & 4 & 52 \\
		EENEENNWNWNNENNWNWWSWWNWNWWSWWSSESESSWSSESEENEESES & 50 & 8 & 58 \\
		EENEENNWWNNENENNWNWWSWWNWNWWSWSSESSWSWSSESEENEESES & 50 & 14 & 64 \\
		EENEENNWWNNENENNWNNWWSSWWNWNWWSWSSESSWSWSSESEENEESES & 52 & 12 & 64 \\
		EENEENNWWNNENENNWNWWSWWNWNWWSWWSSEESSWSWSSESEENEESES & 52 & 12 & 64 \\
		EENEENNWNNEENEENNWNWWSWWNWNWWSWWSSESSWWSWWSSESEENEESES & 54 & 8 & 62 \\
		EENEENNWWNNENENNWNNWWSSWWNWNWWSWWSSEESSWSWSSESEENEESES & 54 & 10 & 64 \\
		EENEENNWWNNENENNWNNWWSSWWNWNWWSWWSSEESSWSWSSESSEENNEESES & 56 & 8 & 64 \\
		\hline
	\end{tabular}
    \centering
	\caption{Admissible loops of length greater than or equal to 8, and the maximal number of extra cells}
\end{sidewaystable}

Unlike larger loops, a loop of length 4 does not have its graph diameter increased if a cell is removed. This leads to the following result.

\begin{corollary}
	We can build any admissible polyomino $P$ with a least 67 cells by adding one by one cell to a maximal path in $P$, without altering the graph diameter at any point.
\end{corollary}

\newpage 
\begin{lemma}
	If $P$ is an admissible polyomino that does not contain any loop of length greater than 4, and $Q$ is a maximal path in $P$, then the radius of $P$ with respect to $Q$ is at most 5.
\end{lemma}

\begin{proof}[Verification]
	There is only one admissible polyomino (up to isometry) that is the union of three paths of graph diameter 6 that only overlap in one common endpoint, shown here.
	\begin{figure}[h]
		\centering
		\includegraphics[width=0.25\linewidth]{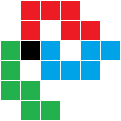}
	\end{figure}

	\noindent However, it contains a loop of length 12, and is not an actual counterexample.
\end{proof}

If $Q$ is an admissible path, suppose its set of cells is partitioned into one or more disjoint subsets
$$Q=\bigcup\limits_{i=1, ..., k} Q_i$$
It seems natural to restrict our attention to the cases in which each $Q_i$ is connected, although this is not strictly required. Let $f(i)$ denote the maximal number of cells that can be added to $Q$ by the following iterative steps, assuming that a cell can only be added to an admissible polyomino if the resulting polyomino is also admissible, and if the graph diameter is not altered. \\

Step 1: Add only cells that are adjacent to at least one cell in $Q_i$

Step $j \in \{ 2, 3, ... \}$: Add only cells that are adjacent to at least one

cell that was added in step $j-1$\\

\noindent It follows from Lemma 2 that five iterative steps is sufficient if the resulting polyomino does not contain any loop of length greater than 4. Ideally, we want to use $k=1$ and find $f(1)$, the exact number of cells that can be added, but this can be time consuming. For higher values of $k$, an upper bound for the number of cells that can be added is given by
$$f(1) + f(2) + ... + f(k)$$
which is often good enough if we are only interested in the global maximum. A combination of the two approaches is also possible: Using upper bounds, we can determine which paths are good \textit{candidates} for creating large admissible polyominoes, and then we can run a full analysis on those.

\newpage \section{Results}
We do not need to go through all paths. For example, it can be verified that if $Q$ contains 48 cells, and we take the first 16, the middle 16 and the last 16 cells as the $k=3$ subsets, then we have $f(1) \leq 9$, $f(2) \leq 8$ and $f(3) \leq 9$. By Lemma 2, this gives us valid upper bounds also for other paths partitioned into segments of 16 cells in a similar way (i.e., 9 at the ends and 8 everywhere in between), and we can cover all diameters from 47 to 90. It has been verified that for any graph diameter less than 106, we have an upper bound for the total number of cells that is less than 142.

For each value of the graph diameter from 106 and upwards, the exact maximal number of cells has been determined.

\begin{center}
	\centering
	\begin{tabular}{|c|c|c|}
		\hline
		Graph & No. of admissible & Maximal size of an \\
		diameter & paths (up to isometry) & admissible polyomino \\
		\hline
		119 & 6 & 138 \\
		118 & 30 & 138 \\
		117 & 55 & 140 \\
		116 & 75 & 141 \\
		115 & 117 & 142 \\
		114 & 144 & 142 \\
		113 & 187 & 142 \\
		112 & 221 & 141 \\
		111 & 266 & 140 \\
		110 & 332 & 138 \\
		109 & 478 & 136 \\
		108 & 679 & 134 \\
		107 & 963 & 133 \\
		106 & 1308 & 132 \\
		\hline
	\end{tabular}
\end{center}

\newpage All the admissible polyominoes with 142 cells can be generated by including exactly one cell of each colour other than black in the figure (provided, of course, that it remains connected), and admissible polyominoes with maximal size for graph diameters 108 through 112 can be obtained by "pruning" them.

\begin{figure}[h]
	\centering
	\includegraphics[width=0.86\linewidth]{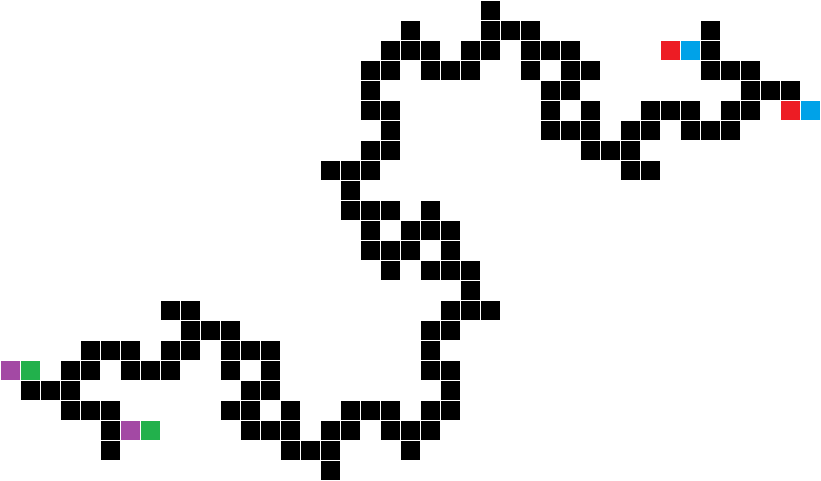}
\end{figure}

Likewise, the largest admissible polyominoes having the maximal graph diameter of 119 can be generated by including exactly one cell of each colour other than black in the next figure.

\begin{figure}[h]
	\centering
	\includegraphics[width=0.9\linewidth]{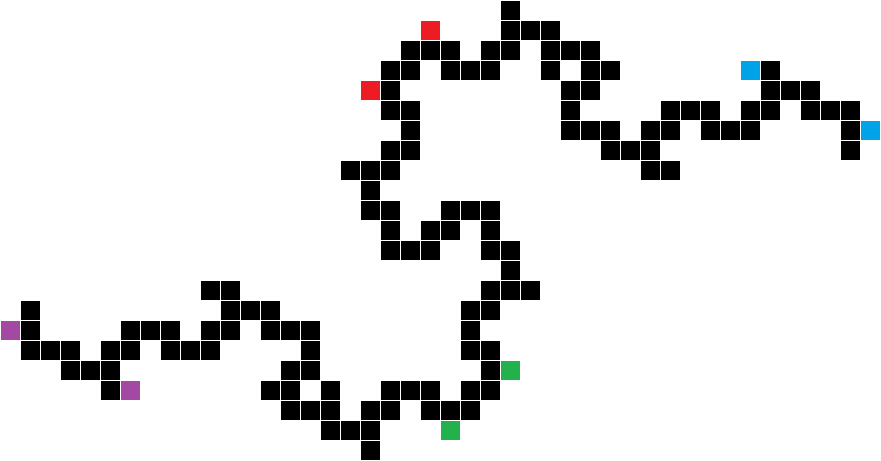}
\end{figure}

\end{document}